\documentclass{amsart}

\usepackage[utf8]{inputenc}
\usepackage[english]{babel}
\newtheorem{theorem}{Theorem}[section]

\newtheorem{lemma}[theorem]{Lemma}

\DeclareMathAlphabet{\mathpzc}{OT1}{pzc}{m}{it}
\usepackage{mathtools}
\usepackage{amsmath}
\usepackage{amsfonts}
\usepackage{blindtext}
\usepackage{hyperref}
\usepackage{url}
\usepackage{amssymb}
\usepackage{bm}
\usepackage{graphicx}

\graphicspath{ {./images/} }
\newcommand{\xv}{\textbf{x}}
\newcommand{\yv}{\textbf{y}}
\newcommand{\zv}{\textbf{z}}
\newcommand{\wv}{\textbf{w}}

\newcommand{\xt}{\textbf{x}^T}
\newcommand{\m}{\mathfrak{m}}
\newcommand{\be}{\bm{\beta}}
\newcommand{\al}{\alpha}
\newcommand{\La}{\Lambda}
\newcommand{\abs}[1]{\displaystyle\left\lvert #1 \right\rvert}

\begin{document}
\title{Quadratic Forms in Prime Variables with small off-diagonal ranks}
\author{Jakub Dobrowolski}
\begin{abstract}
The main goal of this note is to establish the limits of L. Zhao's techniques for counting solutions to quadratic forms in prime variables. Zhao considered forms with rank at least 9, and showed that these equations have solutions in primes provided there are no local obstructions. We consider in detail the degenerate cases of off-diagonal rank 1 and 2, and improve the rank lower bounds to at least 6 and at least 8 respectively. These results complement a recent breakthrough of Green on the non-degenerate rank 8 case.
\end{abstract}
\maketitle
\section{Introduction}
This entire paper concentrates on a problem of finding solutions to the equation
\begin{equation}\label{equation}
\textbf{x}^T A \textbf{x}=t    
\end{equation}
where $A\in M_{n}(\mathbb{Z})$ is a symmetric matrix, $t\in\mathbb{Z}$ and $\textbf{x}$ is a vector with all entries being prime. This problem was first looked at by Hua \cite{Hua}. He proved that in the case of $A$ an identity matrix, $n=5$ and $t\equiv 5$ (mod 24) sufficiently large we can always find a solution to equation (1) in primes. Not long ago Jianya Liu \cite{Liu} was able to handle the more general case. For $A$ satisfying $\operatorname{rank}(A)\geq 10$ and some mild conditions, he was able to prove that an asymptotic number of solutions to (1) in primes in a box $[0,X]^n$ is at least $C_AX^{n-2}(\log X)^{-n}$ where $C_A>0$ is some constant. 
Lilu Zhao \cite{Zhao} was able to prove that $\operatorname{rank}(A)\geq 9$ suffices for all regular matrices $A$. We will be mostly using his techniques throughout this paper. Most recently Ben Green \cite{Green} was able to prove using techniques from representation theory that under some mild non-degeneracy conditions if $\operatorname{rank}(A)=8$ and $A$ has an off-diagonal submatrix of rank 4 then we can find $C_AX^{6}(\log X)^{-8}$ solutions to ({\upshape\ref{equation}}) in a box $[0,X]^8$ with $C_A>0$.\\
\indent In this paper we will be working with 
\begin{equation}
    \mathcal{N}_{A,t}(X)=\sum_{\substack{1\leq x_1,...,x_n\leq X\\ \xt A\xv=t}} \prod_{i=1}^n \La(x_i)
\end{equation}
where $\La(x)$ is the von Mangoldt function and $\mathcal{N}_{A,t}$ is calculating number of solutions to ({\upshape\ref{equation}}) in prime powers in $[0,X]^n$ box. This weighting is rather standard and makes the sum estimates much simpler. Finally we will write $\operatorname{rank}_{\operatorname{off}}(A)=k$ if the biggest rank of a submatrix of $A$ which does not contain diagonal entries is exactly $k$. We can now state main theorem of this paper.\newpage
\begin{theorem}
Let $\mathfrak{S}(A,t)$ be defined as in ({\upshape\ref{sing}}) and $A$ be an indefinite matrix satisfying one of the following:\\
(i) $\operatorname{rank}(A)\geq 6$ and $\operatorname{rank}_{\operatorname{off}}(A)=1$;\\
(ii) $\operatorname{rank}(A)\geq 8$ and $\operatorname{rank}_{\operatorname{off}}(A)=2$.\\ Then:
    \begin{equation}
        \mathcal{N}_{A,t}(X) =\mathfrak{S}(A,t)X^{n-2}+O_K(X^{n-2}(\log X)^{-K})
    \end{equation}
    where $K$ is an arbitrarily large constant.
\end{theorem}
This result was probably known to Zhao but he didn't explicitly write down the details. With this theorem we establish the limitations of his methods for off-diagonal ranks 1 and 2.\\
\indent It's also important to point out how Theorem 1.1 relates to the result of Ben Green\cite{Green}. For $A\in GL_{8}(\mathbb{Z})$, $0\leq r=\operatorname{rank}_{\operatorname{off}}(A)\leq 4$. The case $r=0$ is a result due to Hua\cite{Hua} and for $r=1,2$\, Theorem 1.1 applies. Under mild non-degeneracy conditions the result of Ben Green deals with case $r=4$. Thus there's only case $r=3$ left to improve the restriction on $A$ to $\operatorname{rank}(A)\geq 8$.
\section{Notation}
Notation in this paper is mostly standard. We write $e(z)$ for $e^{2\pi iz}$ and we will let $X$ be a sufficiently large integer and $L=\log X$. When we write $\lvert \xv \rvert\leq X$ or $\xv\leq X$ we mean that coefficients $x_i$ of $\xv$ satisfy $\lvert x_i\rvert \leq X$ or $x_i\leq X$ respectively.\\
\indent In this paper matrix $A$ and matrices $A_i$ will always be symmetric.
We will write $M_{n,k}(S)$ for the set of $n$ by $k$ matrices with entries in $S$ and $M_n(S)$ for the set of $n$ by $n$ matrices with entries in $S$. Finally as usual we denote by $GL_n(S)$ the set of $n$ by $n$ invertible matrices with entries in $S$. \\
\indent For $\xt=(x_1,x_2,...,x_n)$ we will write $\La(\xv)=\La(x_1)\La(x_2)...\La(x_n)$.\\
Many times in this paper we will be summing over the number of integer solutions of certain linear and bilinear equations in some box. Instead of writing that as
\begin{equation}
    \sum_{\substack{equation 1\\ equation 2\\ equation3\\}}1
\end{equation}
we will write it as 
\begin{equation}
    \Bigg\{\substack{equation 1\\ equation 2\\ equation 3}\Bigg\}
\end{equation}
to put a focus on the equations themselves. It will be clear from context which unknowns are actually variables and all of them will be lying in the box $[CX,C'X]$ for $C,C' $ some real constants.
\section{General approach to the problem}
As usual in these kinds of problems we introduce the corresponding exponential sum
\begin{equation}
    S(\alpha)=\sum_{1\leq \textbf{x}\leq X} \Lambda(\textbf{x})e(\alpha \textbf{x}^T A \textbf{x})
\end{equation}
and note that
\begin{equation}
    \int_0^1 S(\alpha)e(-\alpha t)=\sum_{\substack{1\leq \textbf{x}\leq X \\ \xt A \xv=t}} \Lambda(\textbf{x})=\mathcal{N}_{A,t}(X).
\end{equation}
For $a,q$ coprime we will let:
\begin{equation}
    \mathcal{M} (a,q;Q)=\{\al :\lvert \al-\tfrac{a}{q}\rvert\leq \tfrac{Q}{qX^2} \}.
\end{equation}
For fixed $Q\leq X/2$ sets $\mathcal{M} (a,q;Q)$ are disjoint for different $a$ and $q$ provided that $q\leq Q$. We can then write
\begin{equation}
    \mathcal{M}(Q)=\bigcup_{q=1}^Q\bigcup_{\substack{a=1 \\ (a,q)=1}}^q\mathcal{M}(a,q;Q)
\end{equation}
We define major arcs to be $\mathfrak{M}=\mathcal{M}(P)$ where $P=L^K$. In this paper $K$ and $K'$ will denote arbitrarily large parameters independent of $X$, satisfying $K>K'$, which may change from line to line. \\
Finally we will let $\mathfrak{m}=[X^{-1},1+X^{-1}]\setminus \mathfrak{M}$ be the minor arcs.\\
\indent Let's define a few quantities:
\vspace{2mm}
\begin{equation}
C(q,a)=\sum_{\substack{1\leq \mathbf{h}\leq q\\ (\mathbf{h},q)=1}} e\Big(\mathbf{h}^T A\mathbf{h}\frac{a}{q}\Big),    
\end{equation}
\begin{equation}\label{sing}
    \mathfrak{S}(A,t)=\sum_{q=1}^{\infty}\frac{1}{\varphi (q)}\sum_{\substack{1\leq a\leq q\\ (a,q)=1}}C(q,a)e\Big(-\frac{at}{q}\Big).
\end{equation}
In the literature $\mathfrak{S}(A,t)$ is referred to as the singular series. In \cite{Zhao} (Lemma 3.6) Zhao proved that for invertible, indefinite $A$ with $\operatorname{rank}(A)\geq 5$ major arcs bring the expected main contribution in $\int_0^1 S(\alpha)e(-\alpha t)\, d\al$. Namely that:
\begin{equation}
    \int_{\mathfrak{M}} S(\alpha)e(-\alpha t)\, d\al=\mathfrak{S}(A,t)X^{n-2}+O(X^{n-2}L^{-K})
\end{equation}
if equation (1) has positive real solutions and where $\mathfrak{S}(A,t)\gg 1$ provided that $\xt A\xv =t$ has local solutions in the set of $p$-adic units for all primes $p$.\\
\indent So all that's left to do is to estimate the integral over the minor arcs. Our goal is to show that for certain classes of matrices 
\begin{equation}
    \int_{\mathfrak{m}} S(\alpha)e(-\alpha t)\, d\al\ll X^{n-2}L^{-K}.
\end{equation}
This, combined with the result of Zhao about major arcs, will give us the following:
\begin{equation}
     \int_{[0,1]} S(\alpha)e(-\alpha t)\, d\al= \mathfrak{S}(A,t)X^{n-2}+O(X^{n-2}L^{-K})\gg X^{n-2}
\end{equation}
under all the local conditions mentioned above and hence proving Theorem 1.1.\\
\indent Let's look at the trivial bound for the integral over the minor arcs first.
\begin{equation}
    \abs{\int_{\mathfrak{m}} S(\alpha)e(-\alpha t)\, d\al}\leq \sup_{\al\in \mathfrak{m}} \abs{S(\alpha)}\leq \sum_{\xv \leq X} \La(\xv)\leq L^n X^n.
\end{equation}
Let's note that we're still an arbitrarily large power of $\log$ and power of $X^2$ off of our goal in (13).
So the goal for this paper is to save this much over the trivial bound above.\\
\indent To get arbitrarily large logarithm savings we will use a standard result:
\begin{lemma}\label{lemvaug}
Let $\al\in\mathfrak{m}$, $\beta\in\mathbb{R}$ and $d\in \mathbb{Q}\setminus\{0\}$, then we have
\begin{equation}
    \Big \lvert \sum_{x\leq X}\La(x)e(\alpha dx^2+\beta x) \Big\rvert \ll_{d,K} XL^{-K}.
\end{equation}
\end{lemma} 
The proof of it can be found in \cite{Hua2} as Theorem 10.
We now state some lemmas which are the current tool set for saving power of $X^2$ in the integral over the minor arcs. They can be found with proofs in \cite{Zhao} as Lemmas 5.6, 5.7, 5.8 respectively.
\begin{lemma}\label{lem5}
Let $C\in M_n(\mathbb{Q})$ be a symmetric matrix and $H\in M_{n,k}(\mathbb{Q})$. For $
\al \in \mathbb{R}$ and $\be\in \mathbb{R}^k $, we let:
\begin{equation}
    \mathcal{F}(\al,\be)=\sum_{\xv\in \mathcal{X}} \La(\xv)e(\al \xt C\xv +\xt H \be)
\end{equation}
for $\mathcal{X}\subset \mathbb{Z}^n$ a finite subset. Let
\begin{equation}
    \mathcal{N}(\mathcal{F})=\sum_{\substack{\xv\in\mathcal{X} \yv \in \mathcal{X}\\\xt C\xv=\yv^T C\yv\\ \xt H=\yv^T H}}\La(\xv)\La(\yv)
\end{equation}
Then
\begin{equation}
    \int_{[0,1]^{k+1}} \abs{\mathcal{F}(\al,\be)}^2 \,d\al\, d\be \ll \mathcal{N}(\mathcal{F})
\end{equation}
with the implied constant depending only on the matrices $C$ and $H$.
\end{lemma}
For the proof of Lemma 3.2 we simply expand $\abs{\mathcal{F}(\al,\be)}^2$ into a double sum and use orthogonal property of the exponential.
\begin{lemma}\label{lem6}
With $C$ and $H$ as above one has 
\begin{equation}
       \Bigg \{\substack{\xt C\xv=\yv^T C\yv\\ \xt H=\yv^T H}\Bigg\}\ll \Bigg\{\substack{\xt C\yv=0\\ \xt H=0}\Bigg\}
\end{equation}
\end{lemma}
\vspace{2mm}
The proof of this Lemma is just a clever change of coordinates $\mathbf{u}=\xv+\yv$ and $\mathbf{v}=\xv-\yv$.
\begin{lemma}\label{lem7}
If $C\in M_{n,k}(\mathbb{Q})$ and $\operatorname{rank}(C)\geq 2$ then 
\begin{equation}
    \Big\{\substack{\xt C\yv=0}\Big\}\ll LX^{k+n-2}
\end{equation}
with the implied constant depending only on $C$.
\end{lemma}
This is the key step to Zhao's estimates. Let's note that we can trivially bound $\big\{\substack{\xt C\yv=0}\big\}$ by $X^{k+n}$. So we're getting $X^2$ saving over the trivial estimate with this lemma. \\
\indent Let's briefly discuss the use of mentioned lemmas. In all of the problems in section 4 we will be dealing with integrals of the form:
\begin{equation}
    \int_{[0,1]^{k+1}} \sum_{\xv\leq X} \La(\xv)e(\al \xt C\xv +\xt H \be) \,d\al\, d\be
\end{equation}
where $C\in M_n(\mathbb{Q})$ satisfies $\operatorname{rank}(C)\geq 2$. We can apply the Cauchy-Schwarz inequality to get it to the form in Lemma 3.2 with.
So we will be able to trivially estimate
\begin{equation}
    \mathcal{N}(\mathcal{F})\ll L^{2n} \Bigg \{\substack{\xt C\xv=\yv^T C\yv\\ \xt H=\yv^T H}\Bigg\}
\end{equation}
and apply Lemma 3.3. From this, after some change of coordinates we will apply Lemma 3.4. Notice what would happen if we tried to estimate the integral in lemma 3.2 in a more crude way.
\begin{equation}
\abs{\int_{[0,1]^{k+1}} \sum_{\xv\leq X} \La(\xv)e(\al \xt C\xv +\xt H \be) \,d\al\, d\be}\leq \sum_{\xv\leq X} \La(\xv)\ll  L^n X^n
\end{equation}
On the other hand after applying Lemma 3.2-3.4 we get a big improvement. The\vspace{1mm}
saving in this case is of order $L^{O(1)}X^{(n-k-2)/2}$ over the trivial estimate. Since in most of the cases we'll be working with $k\leq n$ this will bring us closer to the desired $X^2$ saving that we need in our problem.
\section{Optimising Zhao's results}
\subsection{Off-diagonal rank 1}
Firstly, as in Zhao's paper, we need to work out the form of the matrix $A$. As  $\operatorname{rank_{\operatorname{off}}}(A)=1$, we know $A$ is not diagonal. So suppose $a_{1,2}\neq 0$. Then we have:
\begin{lemma}
If $A$ satisfies $\operatorname{rank_{\operatorname{off}}}(A)=1$ then $A$ can be written as 
\begin{equation}
   A= \begin{pmatrix}
a&\bm{\xi}^T\\
\bm{\xi}& D+h \bm{\xi} \bm{\xi}^T
\end{pmatrix}
\end{equation}
where $D=\operatorname{diag}(d_1$,...,$d_{n-1}$) with $d_i\in \mathbb{Q}$, $h\in \mathbb{Q}$, $a\in \mathbb{Z}$ and $\textbf{0}\neq \bm{\xi}\in\mathbb{Z}^{n-1}$.
\end{lemma}
\begin{proof}
Let's define $c_i=a_{1,i+1}$\, for $i=1,...,n-1$.
We begin by noting that as $c_1\neq0$, we have $\bm{\xi}^T\vcentcolon= (c_1,...,c_{n-1})\neq\textbf{0}$. Now let's define for $3\leq j\leq n$, \, $\bm{\eta}_j=(a_{j,2},...,a_{j,j-1},a_{j,j+1},...,a_{j,n})^T$, \,$\bm{\xi}_j=(a_{1,2},...a_{1,j-1},a_{1,j+1},...,a_{1,n})^T$ and $\bm{\eta}_2=(a_{2,3},a_{2,4},...,a_{2,n})^T$, $\bm{\xi}_2=(a_{1,3},a_{1,4},...,a_{1,n})^T$. We know that $\operatorname{rank_{\operatorname{off}}}(A)=1$ so for each $j$,\, $\bm{\eta}_j$ is a rational multiple of $\bm{\xi}_j$ as both of those are vectors not on the diagonal.
Therefore there is exactly one $a_{j,j}'\in \mathbb{Q}$ and $h_j\in\mathbb{Q}$ such that $(a_{j,2},...,a_{j,j-1},a_{j,j}',a_{j,j+1},...,a_{j,n})=h_j\bm{\xi}^T$.
Let's write 
\begin{equation}
      A= \begin{pmatrix}
a&\bm{\xi}^T\\
\bm{\xi}& A_1
\end{pmatrix} 
\end{equation}
and let $\textbf{h}^T=(h_2,...,h_n)$. By comparing coefficients we note that $\bm{h} \bm{\xi}^T-A_1$ must be a diagonal matrix, call it $D$. Finally $\bm{h} \bm{\xi}^T$ is a symmetric matrix so for each $1\leq k,\,l\leq n-1$ we have $c_kh_l=c_lh_k$. As $c_1\neq 0$ we can write 
\begin{equation}
    h_k=c_k\frac{h_1}{c_1}.
\end{equation}
So $\bm{h}=h\bm{\xi}$ for $h=\tfrac{h_1}{c_1}\in\mathbb{Q}$ and we get the desired form.
\end{proof}
\begin{lemma}
Let A be given by the form above with $\operatorname{rank}(A)\geq 6$ and let\\ $c_i=a_{1,i+1}$. We can find pairwise distinct $b_i$ with
$1\leq b_1,b_2,b_3,b_4,b_5\leq n-1$ such that $c_{b_1}d_{b_2}d_{b_3}d_{b_4}d_{b_5}\neq 0$.
\end{lemma}
\begin{proof}
By $\operatorname{rank}(A)\geq 6$ we have $\operatorname{rank}(D)\geq 4$. We also assumed w.l.o.g. that $c_1\neq 0$ so let's choose $b_1=1$. Suppose that the above doesn't hold. Namely, for any quadruple
$k_1,k_2,k_3,k_4\geq 2$ we have $d_{k_1}d_{k_2}d_{k_3}d_{k_4}=0$. This gives the upper bound $\operatorname{rank}(D)\leq 4$, so we only have the case of $\operatorname{rank}(D)=4$\, left to consider.\\
\indent We may suppose that $d_1d_2d_3d_4\neq 0$ and $d_j=0$ for $j> 4$. But then by looking at the rank of the whole matrix we conclude that $\operatorname{rank}(c_5, c_6,...,c_{n-1})\neq 0$. That means $c_j\neq 0$ for some $j\geq5$ giving us a quintuple $1,2,3,4,j$.
\end{proof}
With the desired form we can move on to the estimates of $S(\alpha)$ over the minor arcs.
\begin{lemma}
Let $A\in M_n(\mathbb{Z})$ satisfy $\operatorname{rank_{\operatorname{off}}}(A)=1$ with $n\geq6 $, then\newline
$\int_{\mathfrak{m}} \abs{S(\alpha )}\,d\alpha\ll X^{n-2} L^{-K}$    
\end{lemma}
\begin{proof}
 We start by noting that we may use lemmas above to read information on the form of $A$. Then we can write $S(\al)$ as below:
 \begin{equation}
    S(\al)= \sum_{x \leq X} \sum_{\substack{\yv\in\mathbb{N}^{n-1}\\ \yv\leq X}}\La(x)\La(\yv)e(\alpha(ax^2 +2 x\bm{\xi}^T\yv +\yv^TD\yv +h\yv^T\bm{\xi}\bm{\xi}^T\yv)).
 \end{equation}
 For each $\yv\leq X$ we will let $z=\bm{\xi}^T\yv$. We can use orthogonality property of the exponential function to write
\begin{equation}
\begin{aligned}
    S(\alpha)=&\int_0^1 \sum_{\lvert z\rvert \ll X}\sum_{x \leq X} \sum_{\yv \leq X}\La(\yv)\La(x)e(\alpha (ax^2 +2 xz +\yv^TD\yv +hz^2))\\
    &\times e((\bm{\xi}\yv^T-z)\beta)\,d\beta
\end{aligned}
 \end{equation}
 Let's introduce new sums using Zhao's notation.
 \begin{equation}
 \begin{aligned}
        &\mathcal{F}(\alpha,\bm{\beta})=  \sum_{\lvert z\rvert \ll X}  \sum_{x\leq X } \La(x) e(\alpha(ax^2+2 xz+hz^2)- z \beta)\\
   & f_j(\alpha,\bm{\beta})=\sum_{y_j\leq X}\La(y_j)e(\alpha d_jy_j^2+y_jc_j\beta)
 \end{aligned}
\end{equation}
 for $j=1,...,n-1$. By the above lemma we can find a quintuple $b_1,b_2,b_3,b_4,b_5$ and suppose for notational simplicity that it's $1,2,3,4,5$ i.e. $c_1d_2d_3d_4d_5\neq 0$. We bound the integral using the Cauchy-Schwarz inequality.
\begin{equation}\label{prod}
    \int_{\mathfrak{m}} \abs{S(\alpha )}\,d\alpha\leq (\mathcal{I}_1 \mathcal{I}_2)^{\tfrac{1}{2}}\prod_{i=5}^{n-1} \sup_{\substack{\alpha \in \m\\ \beta\in [0,1]}}\Big \lvert\sum_{y_i\leq X}\La(y_i)e(\alpha d_i y_i^2+c_iy_i\beta)\Big \rvert  
\end{equation}
where
\begin{equation}
    \mathcal{I}_1=\int_{[0,1]^2} \Big \lvert \mathcal{F}(\al,\beta)f_4(\al,\beta)\Big \rvert^2 \,d\beta \,d\alpha
\end{equation}
and 
\begin{equation}
    \mathcal{I}_2=\int_{[0,1]^2} \Big \lvert f_1(\al,\beta) f_2(\al,\beta) f_3(\al,\beta)\Big \rvert^2 \,d\beta \,d\alpha.
\end{equation}
As $d_5\neq 0$ we can bound the product in ({\upshape\ref{prod}}) using Lemma 3.1 by
\begin{equation}
    \prod_{i=5}^{n-1} \sup_{\substack{\alpha \in \m\\ \beta\in [0,1]}}\Big \lvert\sum_{y_i\leq X}\La(y_i)e(\alpha d_i y_i^2+c_iy_i\beta)\Big \rvert  \ll X^{n-5}L^{-K}
\end{equation}
where for $i\geq 6$ we used a trivial bound on the sum: $LX$. \\
\indent Now we turn to the two integrals. Note that by Lemma 3.2
\begin{equation}
    \mathcal{I}_1 \ll L^{4} \Bigg\{\substack{ax^2+2 xz+hz^2+d_4y_4^2=ax^{\prime 2}+ 2 x'z'+hz^{\prime 2}+d_4y_4^{\prime 2}\\ c_4y_4-z=c_4y_4'-z'}\Bigg\}.
\end{equation}
We can therefore apply Lemma 3.3 to get 
\begin{equation}
    \mathcal{I}_1 \ll L^{4} \Bigg\{\substack{axx'+ xz'+ x'z+hzz'+d_4y_4y_4'=0\\ c_4y_4-z=0}\Bigg\}\\
\end{equation}
and if we plug in $z=c_4 y_4$ note that we get a bilinear equation in 5 variables. 
\begin{equation}
\mathcal{I}_1  \ll L^{4} \Bigg\{axx'+ xz'+ c_4y_4x'+hc_4y_4z'+d_4y_4y_4'=0\Bigg\}
\end{equation}
The corresponding matrix is of rank at least 2 as coefficients by $xz'$ and $y_4y_4'$ are nonzero. So by Lemma 3.4 we can bound the last sum by $L^5 X^3$.
We do the same procedure for the second integral.
\begin{equation}
    \mathcal{I}_2=\int_{[0,1]^2} \Big \lvert f_1(\al,\beta) f_2(\al,\beta) f_3(\al,\beta)\Big \rvert^2 \,d\beta \,d\alpha \ll L^{6} \Bigg\{\substack{d_1y_1y_1'+d_2y_2y_2'+d_3y_3y_3'=0\\ c_1y_1+c_2y_2+c_3y_3=0}\Bigg\}
\end{equation}
where we already used Lemma 3.2 and 3.3. Now since $c_1\neq 0$ we can divide by it and write $y_1$ in terms of $y_2$ and $y_3$. So we get:
\begin{equation}
    \mathcal{I}_2\ll L^{6} \Bigg\{-d_1y_1'(c_2y_2+c_3y_3)/c_1+d_2y_2y_2'+d_3y_3y_3'=0\Bigg\}
\end{equation}
a bilinear equation in 5 variables with rank of a corresponding matrix at least two as this time $d_2 d_3\neq 0$. So we can apply Lemma 3.4 to get 
\begin{equation}
    \mathcal{I}_2\ll L^{7}X^3.
\end{equation}
 We combine the results to get that gives us 
 \begin{equation}
     \int_{\mathfrak{m}} \abs{S(\alpha )}\,d\alpha\ll L^{-K} X^{n-5} L^{5/2} X^{3/2}L^{5/2}X^{3/2} \ll L^{-K}X^{n-2} 
 \end{equation}
 as we wanted.
\end{proof}
\subsection{Off-diagonal rank 2}
Suppose now without loss of generality that\\ $\operatorname{rank}_{\operatorname{off}}(A)=\operatorname{rank}(B)=2$ where 
$B=
\begin{pmatrix}
a_{1,3}&a_{1,4}\\
a_{2,3}&a_{2,4}
\end{pmatrix}
$.
We consider different cases depending on ranks of following matrices:
\begin{align}
    B_1=
\begin{pmatrix}
a_{1,3}&a_{1,5}&a_{1,6}&.&.&.&a_{1,n}\\
a_{2,3}&a_{2,5}&a_{2,6}&.&.&.&a_{2,n}
\end{pmatrix}\\
 B_2=
\begin{pmatrix}
a_{1,4}&a_{1,5}&a_{1,6}&.&.&.&a_{1,n}\\
a_{2,4}&a_{2,5}&a_{2,6}&.&.&.&a_{2,n}
\end{pmatrix}.
\end{align}
There are three cases to consider (up to reordering of the matrix $A$) and so three lemmas which we will need to prove:
\begin{lemma}\label{lem1}
If $\operatorname{rank}(B_1)=\operatorname{rank}(B_2)=1$ and $\operatorname{rank}(A)\geq 8$ then we have\\
$\int_{\mathfrak{m}} \abs{S(\alpha )}\,d\alpha\ll X^{n-2} L^{-K}$  
\end{lemma}
\begin{lemma}\label{lem2}
If $\operatorname{rank}(B_1)=2, \operatorname{rank}(B_2)=1$ and $\operatorname{rank}(A)\geq 8$ then we have\\
$\int_{\mathfrak{m}} \abs{S(\alpha )}\,d\alpha\ll X^{n-2} L^{-K}$  
\end{lemma}
\begin{lemma}\label{lem3}
If $\operatorname{rank}(B_1)=\operatorname{rank}(B_2)=2$ and $\operatorname{rank}(A)\geq 8 $ then we have\\
$\int_{\mathfrak{m}} \abs{S(\alpha )}\,d\alpha\ll X^{n-2} L^{-K}$ . 
\end{lemma}
Strategy for all three lemmas is similar as for proof of off-diagonal rank 1 case. In each case we will firstly set up a structure lemma for the form of matrix $A$. With that we will introduce a new variable which will allow us to bound the integral using the Cauchy-Schwarz inequality. In the end we will use Lemmas 3.1-3.4 to get the desired result. 
\subsubsection{Proof of Lemma 4.4}
\begin{lemma}
If $A$ is a matrix with $\operatorname{rank}(B_1)=\operatorname{rank}(B_2)=1$ then $A$ is of the form \\
\begin{center}
$
    A= \begin{pmatrix}
A_1&B&0\\
B^T&A_2&C\\
0&C^T&D
\end{pmatrix}
$
\end{center} 
where $B\in GL_2(\mathbb{Z})$, $C\in M_{2,n-4}(\mathbb{Z})$ and $D=\operatorname{diag}(d_1,...,d_{n-4})$
\end{lemma}
\begin{proof}
Let's introduce vectors $\gamma_j=(a_{1,2+j},a_{2,2+j})^T$ for $j=1,2,...,n-2$. By looking at ranks of $B_1$ and $B_2$ we note that for any $j\geq 3$\, $\operatorname{rank}(\gamma_1,\gamma_j)=$ $\operatorname{rank}(\gamma_2,\gamma_j)=1$, but as $\gamma_1$ and $\gamma_2$ are linearly independent we conclude that $\gamma_j=0$ for $j\geq3$.
Finally let's consider matrices \\
\begin{center}
   $ 
M_{i,j}=\begin{pmatrix}
a_{1,3}&a_{1,4}&a_{1,j}\\
a_{2,3}&a_{2,4}&a_{2,j}\\
a_{i,3}&a_{i,4}&a_{i,j}
\end{pmatrix}$  
\end{center}
for $i\neq j$ and $i,j>4$.\\
\indent Each such $M_{i,j}$ is an off-diagonal matrix so it must have rank 2 (as $B$ is its submatrix). Therefore as $a_{1,j}=a_{2,j}=0$ we must have $a_{i,j}=0$, and we get the desired form.
\end{proof}
\begin{proof}[ \,Proof of Lemma~{\upshape\ref{lem1}}]
We now turn to the estimate of $S(\alpha)$. By the structure lemma above we get that:       
\begin{equation}
    S(\alpha)=
    \sum_{\substack{ \textbf{x},\textbf{y}\leq X \\ \textbf{x},\textbf{y}\in \mathbb{N}^2}}
    \sum_{\substack{ \textbf{z}\leq X \\ \textbf{z}\in \mathbb{N}^{n-4}}}\La(\xv) \La(\yv)\La(\zv)
     e(\alpha(\xt A_1\xv +2\xt B\yv+\yv^TA_2 \yv +2\yv^T C \zv+\textbf{z}^T D \textbf{z}))
\end{equation}
As in the proof of Lemma 4.3 for each $\xv,\yv\in\mathbb{N}^2$ and $\zv\in\mathbb{N}^{n-4}$ satisfying $\xv,\yv,\zv\leq X$ we let $\textbf{w}^T=2\xt B+\textbf{y}^T A_2+2\textbf{z}^TC^T$. By orthogonality property we have then
\begin{equation}
    \begin{aligned}
S(\alpha)={} &\int_{[0,1]^2}
    \sum_{\substack{\lvert\textbf{w}\rvert\ll X \\ \textbf{w}\in \mathbb{Z}^2}}
    \sum_{\substack{ \textbf{x}\leq X \\ \textbf{x}\in \mathbb{N}^2}}
    \sum_{\substack{ \textbf{y}\leq X \\ \textbf{y}\in \mathbb{N}^2}}
    \sum_{\substack{ \textbf{z}\leq X \\ \textbf{z}\in \mathbb{N}^{n-4}}}\La(\xv) \La(\yv)\La(\zv)
     e(\alpha(\xt A_1\xv +\wv^T \yv+\textbf{z}^T D \textbf{z}))\\
    &\times e((2\xt B+\textbf{y}^T A_2+2\textbf{z}^TC^T-\textbf{w}^T)\bm{\beta})\, d\bm{\beta}.
\end{aligned}
\end{equation}
where $\bm{\beta}$ is a two dimensional vector with  $d\bm{\beta}= d\beta_1 \, d\beta_2$
We introduce exponential sums.
\begin{equation}
\begin{aligned}
        &\mathcal{F}(\alpha,\bm{\beta})=    \sum_{\substack{ \textbf{x}\leq X \\ \textbf{x}\in \mathbb{N}^2}} \La(\xv)   e(\alpha\xt A_1\xv+2\xt B\bm{\beta}))\\
        & \mathcal{H}_j(\alpha,\bm{\beta})=
        \sum_{\substack{\lvert w_j\rvert\ll X \\ w_j\in \mathbb{Z}}}   \sum_{\substack{ y_j\leq X \\ y_j\in \mathbb{N}}}
        e(\alpha w_jy_j+y_j\bm{\gamma_j}^T\bm{\beta}-w_j\bm{e}_j^T\bm{\beta})\La(y)
\end{aligned}
\end{equation}
where $j=1,2$, $\bm{\gamma}_j^T=(a_{2+j,3},a_{2+j,4})$, $\bm{e}_j$ standard basis vectors of size 2 and $y_j,w_j$ are entries of vectors $\textbf{y}$ and $\textbf{w}$ respectively. Finally we let
\begin{equation}
    f_j(\alpha,\bm{\beta})=\sum_{z_j\leq X }\La(z_j)e(\alpha d_jz_j^2+2z_j\bm{\xi}_j^T\bm{\beta})
\end{equation}
with $\bm{\xi}_j=(a_{3,4+j},a_{4,4+j})^T$ and $z_j$ entries of vector $\zv$ for $j=1,2,...,n-4$.\\
\indent We will consider 2 different cases depending on $\operatorname{rank}(D)$. Let's note that as $\operatorname{rank}(A)\geq 8$ we have that $\operatorname{rank}(D)\geq 2$.
Let's suppose first that $\operatorname{rank}(D)\geq 3$ with $d_1d_2d_3\neq 0$. By applying CS inequality we get:
\begin{equation}
\begin{aligned}
    \int_{\mathfrak{m}} \abs{S(\alpha )}\,d\alpha&\leq 
    \Bigg(\sup_{\substack{\alpha\in \m \\ \bm{\beta}\in [0,1]^2}}\Big \lvert\prod_{i=3}^{n-4} f_i(\alpha,\bm{\beta})\Big \rvert\Bigg)
    \displaystyle\left(\int_{[0,1]^3}\abs{\mathcal{F}(\alpha,\be)f_1(\alpha,\be)f_2(\alpha,\be)}^2 \,d\al \,d\be \right)^{\tfrac{1}{2}}\\
    &\times\displaystyle\left(\int_{[0,1]^3}\abs{\mathcal{H}_1(\alpha,\be)\mathcal{H}_2(\alpha,\be)}^2\,d\al \,d\be \right)^{\tfrac{1}{2}}
\end{aligned}
\end{equation}
As $d_3\neq 0$ in the supremum we can save an arbitrarily large power of log over the trivial bound. Namely:
\begin{equation}
    \sup_{\substack{\alpha\in \m \\ \bm{\beta}\in [0,1]^2}}\abs{\prod_{i=3}^{n-4} f_i(\alpha,\bm{\beta})}\ll X^{n-6}L^{-K}
\end{equation}
by Lemma 3.1.
After applying Lemmas 3.2 and 3.3 the first integral becomes
\begin{equation}
    \int_{[0,1]^3}\abs{\mathcal{F}(\alpha,\be)f_1(\alpha,\be)f_2(\alpha,\be)}^2 \,d\al \,d\be \ll L^{8} \Bigg\{\substack{\xt A_1\xv'+ d_1z_1z_1'+d_2z_2z_2'=0\\ \xt B+z_1\bm{\xi}_1^T+z_2\bm{\xi}_2^T=0}\Bigg\}.
\end{equation}
Let's observe that as $B$ is invertible we can write $\xv$ in terms of $z_1$ and $z_2$. So we can get a bilinear equation in 6 variables and as we chose $d_1,d_2\neq 0$, the matrix corresponding to it is of rank at least 2. So first integral can be bounded, using Lemma 3.4, by 
\begin{equation}
    \ll L^{8} \Bigg\{-(z_1\bm{\xi}_1^T+z_2\bm{\xi}_2^T)B^{-1} A_1\xv'+ d_1z_1z_1'+d_2z_2z_2'=0\Bigg\}
    \ll L^9 X^4.
\end{equation}
\indent So after taking square root we are left with $L^{9/2}X^2$. Note that we've saved a power of $X^2$ over the trivial bound. One power of $X$ came from $B$ being invertible and the other from applying Lemma 3.4. So we hope that the second integral can be bounded from above by $L^{O(1)}X^4$ and indeed it can. Second integral is bounded by:
\begin{equation}
    \ll L^{4} \Bigg\{\substack{w_1y_1'+w_1'y_1+w_2y_2'+w_2'y_2=0\\ y_1\bm{\gamma_1^T}+y_2\bm{\gamma_2^T} -w_1 \bm{e}_1^T-w_2\bm{e}_2^T=0}\Bigg\}.
\end{equation}
Because matrix $( \bm{e}_1,\bm{e}_2)$ is invertible we can write $w_1$ and $w_2$ as a linear combination of $y_1$ and $y_2$. So we get a bound:
\begin{equation}
    \ll
    L^4 \Bigg\{\yv^T (\bm{\gamma}_1,\bm{\gamma}_2)^T\yv'+\yv^T\wv'=0\Bigg\} \ll L^5 X^4
\end{equation}
where for the last inequality we applied Lemma 3.4 again.
Therefore after combining those results we get $\int_{\mathfrak{m}} \abs{S(\alpha )}\,d\alpha\ll X^{n-2} L^{-K}$ \\
\indent Now suppose $\operatorname{rank}(D)=2$ and after permuting suppose $d_1d_2\neq 0$ and $d_k=0$ for $k=3,...,n-4$. Because $\operatorname{rank}(A)\geq 8$ and $\operatorname{rank}(D)=2$  we may find $i,j$ with $1\leq i\leq 2$ and $j\geq 3$ such that $\operatorname{rank}(\bm{e}_i,\bm{\xi}_j)=2$. Without loss of generality let's suppose $i=1,j=3$. In this case we split the starting integral as follows:
\begin{equation}
    \begin{aligned}
    \int_{\mathfrak{m}} \abs{S(\alpha )}\,d\alpha&\leq \Bigg(\sup_{\substack{\alpha\in \m \\ \bm{\beta}\in [0,1]^2}}\Big\lvert\prod_{i\neq 2,3} f_i(\alpha,\bm{\beta})\Big\rvert\Bigg)\displaystyle\left(\int_{[0,1]^3}\abs{\mathcal{F}(\alpha,\be)\mathcal{H}_2(\alpha,\be)}^2 \,d\al \,d\be \right)^{\tfrac{1}{2}}\\
    &\times\displaystyle\left(\int_{[0,1]^3}\abs{\mathcal{H}_1(\alpha,\be)f_2(\alpha,\be)f_3(\alpha,\bm{\beta})}^2\,d\al \,d\be \right)^{\tfrac{1}{2}}
\end{aligned}
\end{equation}
Because $d_1\neq 0$ we again can save an arbitrarily large power of log over the trivial bound. Namely:
\begin{equation}
    \sup_{\substack{\alpha\in \m \\ \bm{\beta}\in [0,1]^2}}\abs{\prod_{i\neq 2,3} f_i(\alpha,\bm{\beta})}\ll X^{n-6}L^{-K}
\end{equation}
by Lemma 3.1.
First integral can be bounded by:
\begin{equation}
    \ll L^{6} \Bigg\{\substack{2\xt A_1\xv'+ w_2y_2'+w_2'y_2=0\\ 
    2\xt B+y_2\bm{\gamma}_2^T-w_2\bm{e}_2^T=0}\Bigg\}
    \ll L^{6} \Bigg\{-(y_2\bm{\gamma}_2^T-w_2\bm{e}_2^T)B^{-1}A_1\xv'+ w_2y_2'+w_2'y_2=0\Bigg\} 
\end{equation}
where we again used the fact that $B$ was invertible to write $\xv$ in terms of $y_2$ and $w_2$. So we have a bilinear equation in 6 variables with corresponding matrix of rank at least 2. So we can apply Lemma 3.4 to say that the first integral is bounded by:
\begin{equation}
    \left(\int_{[0,1]^3}\abs{\mathcal{F}(\alpha,\be)\mathcal{H}_2(\alpha,\be)}^2 \,d\al \,d\be \right)^{\tfrac{1}{2}}\ll (L^7 X^4)^{1/2}=L^{7/2}X^2.
\end{equation}
We again apply Lemmas 3.2 and 3.3 to the second integral in (53) to get that it's bounded by
\begin{equation}
    \ll L^{6} \Bigg\{\substack{w_1y_1'+w_1'y_1+2d_2z_2z_2'+2d_3z_3z_3'=0\\ y_1\bm{\gamma_1^T}-w_1 \bm{e}_1^T+2z_2\bm{\xi}_2^T+2z_3\bm{\xi}_3^T=0}\Bigg\}.
\end{equation}
Because $d_3=0$ we note that variable $z_3'$ doesn't have any restrictions on it, so can take $O(X)$ possible values. Therefore
\begin{equation}
    \int_{[0,1]^3}\abs{\mathcal{H}_1(\alpha,\be)f_2(\alpha,\be)f_3(\alpha,\bm{\beta})}^2\,d\al \,d\be \ll L^6X \Bigg\{\substack{w_1y_1'+w_1'y_1+2d_2z_2z_2'=0\\ y_1\bm{\gamma_1^T}-w_1 \bm{e}_1^T+2z_2\bm{\xi}_2^T+2z_3\bm{\xi}_3^T=0}\Bigg\}.
\end{equation}
By what we said above $(\bm{e}_1,\bm{\xi}_3)$ is an ivertible matrix. So we can write $w_1$ and $z_3$ in terms of $y_1$ and $z_2$. This leaves us with a bilinear equation in 5 variables of rank at least 2 as $d_2\neq 0$ and coefficient by $w_1'y_1$ is non zero. So using Lemma 3.4 again we get:
\begin{equation}
    \displaystyle\left(\int_{[0,1]^3}\abs{\mathcal{H}_1(\alpha,\be)f_2(\alpha,\be)f_3(\alpha,\bm{\beta})}^2\,d\al \,d\be \right)^{\tfrac{1}{2}}\ll (XL^7X^3)^{1/2}=L^{7/2}X^2.
\end{equation}
After combining the result we get the exact expression we need.\\
\end{proof}
\subsubsection{Proof of Lemma 4.5}
\begin{lemma}
If $A$ is a matrix with $\operatorname{rank}(B_1)=2, \operatorname{rank}(B_2)=1$ then $A$ is of the form \\
\begin{center}
$
    A= \begin{pmatrix}
A_1&\bm{\gamma}_1&\bm{\gamma}_2\bm{\xi}^T\\
\bm{\gamma}_1^T&a&\bm{v}^T\\
\bm{\xi}\bm{\gamma}_2^T&\bm{v}&D+h\bm{\xi}\bm{\xi}^T
\end{pmatrix}
$
\end{center} 
where $\bm{\gamma}_1\in \mathbb{Z}^2$, $\bm{\gamma}_2\in \mathbb{Q}^2$, $\bm{\xi},\bm{v}\in \mathbb{Z}^{n-3}$, $h\in \mathbb{Q}$ and $D=\operatorname{diag}(d_1,...,d_{n-3})$. Finally one also has $(\bm{\gamma}_1,\bm{\gamma}_2)\in GL_2(\mathbb{Q})$.
\end{lemma}
\begin{proof}
We begin by noting that as $\operatorname{rank}(B_2$)=1 and $(a_{1,4},a_{2,4})^T\neq0$ we can say that $B_2=(a_{1,4},a_{2,4})^T\bm{\xi}'^T$ with $\bm{\xi}'\in \mathbb{Q}^{n-3}$. So after rescaling appropriately we have
$B_2=\bm{\gamma}_2\bm{\xi}^T$ for some $\bm{\xi}\in \mathbb{Z}^{n-3}$ and $\bm{\gamma}_2$ satisfying $\operatorname{rank}(B)=\operatorname{rank}(\bm{\gamma}_1,\bm{\gamma}_2)=2$.
Let's define for $4\leq j\leq n$, $1\leq i\leq 2$ $\bm{\eta}_j=(a_{j,3},...,a_{j,j-1},a_{j,j+1},...,a_{j,n})^T$,  $\bm{\theta}_{i,j}=(a_{i,3},...a_{i,j-1},a_{i,j+1},...,a_{i,n})^T$. We know that $\operatorname{rank}(B)=\operatorname{rank}(B_1)=2$ so for each $j$, $\bm{\eta}_j$ is a linear combination of $\bm{\theta}_{1,j},\bm{\theta}_{2,j}$. Because of that if we let $\bm{\theta}_{i}=(a_{i,4},...,a_{i,n})^T$ for $i=1,2$ then for each $4\leq j\leq n$ there is exactly one $a_{j,j}'\in \mathbb{Q}$ such that $\bm{\varphi}_j=(a_{j,4},...,a_{j,j-1},a_{j,j}',a_{j,j+1},...,a_{j,n})^T$ is a linear combination of $\bm{\theta}_1$ and $\bm{\theta}_2$. Because $\operatorname{rank}(B_2$)=1 we have that $\bm{\theta}_1$ and $\bm{\theta}_2$ are both a rational multiple of $\bm{\xi}$ i.e. $\bm{\varphi}_j=h_j\bm{\xi} $ for some $h_j\in \mathbb{Q}$. If we now let $\mathbf{h}^T= (h_2,...,h_n)$ then
\begin{center}
$
    A= \begin{pmatrix}
A_1&\bm{\gamma}_1&\bm{\gamma}_2\bm{\xi}^T\\
\bm{\gamma}_1^T&a&\bm{v}^T\\
\bm{\xi}\bm{\gamma}_2^T&\bm{v}&D+\bm{h}\bm{\xi}^T
\end{pmatrix}
$
\end{center} 
for $D$ a diagonal matrix. But as $A$ is symmetric we get $\mathbf{h}=h \bm{\xi}$ for some $h\in \mathbb{Q}$.
\end{proof}
\begin{proof}[ \,Proof of Lemma~{\upshape\ref{lem2}}]
Note that $\operatorname{rank}(D)+\operatorname{rank}(\bm{v})+1+3\geq \operatorname{rank}(A)\geq 8$ which gives us $\operatorname{rank}(D)\geq 3$. If we let $\bm{\xi}^T=(\epsilon_1,...,\epsilon_{n-3})$, then as $\operatorname{rank}(B)=\operatorname{rank}(B_1$)=2 we can find $2\leq l\leq n-3$ such that $\epsilon_1\epsilon_l\neq 0$. As $\operatorname{rank}(D)\geq 3$ we know that $d_k\neq 0$ for some $2\leq k\neq l$. So suppose w.l.o.g. that $l=2$ and $k=3$. Let's now proceed with the estimates.
\begin{equation}
\begin{aligned}
   S(\alpha)=&\sum_{\substack{ \textbf{x}\leq X \\ \textbf{x}\in \mathbb{N}^2}}
    \sum_{\substack{ y\leq X \\ y\in \mathbb{N}}}
    \sum_{\substack{ \textbf{z}\leq X \\ \textbf{z}\in \mathbb{N}^{n-3}}}\La(\xv) \La(y)\La(\zv)
     e(\alpha(\xt A_1\xv +2\xt \bm{\gamma_1}y+ay^2+2\xt \bm{\gamma_2}\bm{\xi}^Tz))\\
     &\times e(\alpha(2y\bm{v}^T\zv+\textbf{z}^T D \textbf{z}+\zv^Th\bm{\xi}\bm{\xi}^T\zv))
\end{aligned}
\end{equation}
For each $\xv\in \mathbb{N}^2$, $y\in \mathbb{N}$ and $\zv\in \mathbb{N}^{n-3}$ satisfying $\xv,y, \zv\leq X$ we introduce new variables $s=\bm{\xi}^T\zv$, $w=2\xt\bm{\gamma_1}+ay+2\zv^T\bm{v}$ and use orthogonality property of exponential function to write $S(\al)$ as
\begin{equation}
        \begin{aligned}
S(\alpha)={} &\int_{[0,1]^2}
    \sum_{\substack{\lvert s\rvert\ll X \\ s\in \mathbb{Z}}}
    \sum_{\substack{\lvert w\rvert\ll X \\ w\in \mathbb{Z}}}
    \sum_{\substack{ \textbf{x}\leq X \\ \textbf{x}\in \mathbb{N}^2}}
    \sum_{\substack{ y\leq X \\ y\in \mathbb{N}}}
    \sum_{\substack{ \textbf{z}\leq X \\ \textbf{z}\in \mathbb{N}^{n-3}}}\La(\xv) \La(y)\La(\zv)
     e(\alpha(\xt A_1\xv +wy+\textbf{z}^T D \textbf{z}))\\
    &\times e(\al(hs^2+2\xt \bm{\gamma}_2s)+(2\xt \bm{\gamma}_1+ay+2\zv^T\bm{v}-w)\beta_1+(\bm{\xi}^T\zv-s)\beta_2) d\bm{\beta}
\end{aligned}
\end{equation}
where $d\be=d\beta_1\, d\beta_2$
As before we will split the integral into different parts. First we introduce different exponential sums.
\begin{equation}
\begin{aligned}
        &\mathcal{F}(\alpha,\bm{\beta})=  \sum_{\substack{\lvert s\rvert\ll X \\ s\in \mathbb{Z}}}  \sum_{\substack{ \textbf{x}\leq X \\ \textbf{x}\in \mathbb{N}^2}} \La(\xv)   e(\alpha(\xt A_1\xv+hs^2+2\xt \bm{\gamma}_2s)+2\xt \bm{\gamma}_1\beta_1-s\beta_2)\\
        & \mathcal{H}(\alpha,\bm{\beta})=
        \sum_{\substack{\lvert w\rvert\ll X \\ w\in \mathbb{Z}}}   \sum_{\substack{ y\leq X \\ y\in \mathbb{N}}}
        e(\alpha wy+ay\beta_1-w\beta_1)\La(y)
\end{aligned}
\end{equation}
and finally for $j=1,2,...,n-3$ we let
\begin{equation}
    f_j(\alpha,\bm{\beta})=\sum_{\substack{ z_j\leq X \\ z_j\in \mathbb{N}}}\La(\zv)e(\alpha d_jz_j^2+2z_jv_j\beta_1+\epsilon_jz_j\beta_2)
\end{equation}
where $\bm{v}^T=(v_1,...,v_{n-3})$.
Let's bound the integral as below.
\begin{equation}
    \begin{aligned}
     \int_{\mathfrak{m}} \abs{S(\alpha )}\,d\alpha\leq 
    &\Bigg(\sup_{\substack{\alpha\in \m \\ \bm{\beta}\in [0,1]^2}}\big\lvert\prod_{k\neq 2,3} f_k(\alpha,\bm{\beta})\big\rvert\Bigg)\displaystyle\left(\int_{[0,1]^3}\abs{\mathcal{F}(\alpha,\be)f_3(\alpha,\bm{\beta})}^2 \,d\al \,d\be \right)^{\tfrac{1}{2}}\\
    &\times\displaystyle\left(\int_{[0,1]^3}\abs{\mathcal{H}(\alpha,\be)f_2(\alpha,\bm{\beta})}^2\,d\al \,d\be \right)^{\tfrac{1}{2}}        
    \end{aligned}
\end{equation}
From our discussion above we know that $d_3\neq 0$. So we can bound the first integral by:
\begin{align}
    &\ll L^{6} \Bigg\{\substack{2\xt A_1\xv'+ 2hss'+2d_3z_3z_3'+\xt\bm{\gamma}_2s'+s\bm{\gamma}_2^T\xv'=0\\ 
    \xt\bm{\gamma}_1+z_3v_3=0\\ s=\epsilon_3 z_3}\Bigg\}\\
    &\ll L^{6} \Bigg\{\substack{2\xt A_1\xv'+ 2h\epsilon_3 z_3s'+2d_3z_3z_3'+\xt\bm{\gamma}_2s'+\epsilon_3z_3\bm{\gamma}_2^T\xv'=0\\ 
    \xt\bm{\gamma}_1+z_3v_3=0}\Bigg\}.
\end{align}
Let's introduce a new variable $q=\xt\bm{\gamma}_2$. We can rewrite the above set of equations as
\begin{equation}
    \Bigg\{\substack{2\xt A_1\xv'+ 2h\epsilon_3 z_3s'+2d_3z_3z_3'+qs'+\epsilon_3z_3\bm{\gamma}_2^T\xv'=0\\ 
    \xt(\bm{\gamma}_1,\bm{\gamma}_2)+(z_3v_3,-q)=0}\Bigg\}.
\end{equation}
Note that as $(\bm{\gamma}_1,\bm{\gamma}_2)$ is invertible we can again write $\xv$ as a combination of $z_3$ and $q$. After that we are left with bilinear equation in 6 variables with rank at least 2 
\begin{equation}
    \Bigg\{-2(z_3v_3,-q)(\bm{\gamma}_1,\bm{\gamma}_2)^{-1} A_1\xv'+ 2h\epsilon_3 z_3s'+2d_3z_3z_3'+qs'+\epsilon_3z_3\bm{\gamma}_2^T\xv'=0\Bigg\}.
\end{equation}
So we again apply Lemma 3.4 to get:
\begin{equation}
    \left(\int_{[0,1]^3}\abs{\mathcal{F}(\alpha,\be)f_3(\alpha,\bm{\beta})}^2 \,d\al \,d\be \right)^{\tfrac{1}{2}}\ll (L^7 X^4)^{1/2}=L^{7/2}X^2.
\end{equation}
For the second integral we use Lemmas 3.2 and 3.3 to arrive at:
\begin{equation}
    \int_{[0,1]^3}\abs{\mathcal{H}(\alpha,\be)f_2(\alpha,\bm{\beta})}^2\,d\al \,d\be \ll L^{4} \Bigg\{\substack{2d_2 z_2z_2'+wy'+w'y=0\\ 
    2z_2v_2+ya-w=0\\ \epsilon_2 z_2=0}\Bigg\}.
\end{equation}
As $\epsilon_2\neq 0$ we must have $z_2=0$. Leaving us with:
\begin{equation}
    L^4\Bigg\{\substack{wy'+w'y=0\\ 
    ya=w}\Bigg\}\ll L^4 \Bigg\{y(ay'+w')=0\Bigg\}\ll L^4 X^2.
\end{equation}
Finally as $\operatorname{rank}(D)\geq 3$ we can find $k\neq 2,3$ such that $d_k\neq 0$. Therefore by Lemma 3.1 we get
\begin{equation}
    \sup_{\substack{\alpha\in \m \\ \bm{\beta}\in [0,1]^2}}\abs{\prod_{k\neq 2,3} f_k(\alpha,\bm{\beta})}\ll L^{-K}X^{n-5}
\end{equation}
After combining all three estimates we get the desired result.
\end{proof}
\subsubsection{Proof of Lemma 4.6}
\begin{lemma}\label{lem8}
Let $\operatorname{rank}(B_1)=\operatorname{rank}(B_2)=2$. Then $A$ is of the form:
\begin{center}
$
    A= \begin{pmatrix}
A_1&(\bm{\gamma}_1,\bm{\gamma}_2)C\\
C^T(\bm{\gamma}_1,\bm{\gamma}_2)^T&D+C^THC
\end{pmatrix}
$
\end{center} 
where $(\bm{\gamma}_1,\bm{\gamma}_2)\in GL_2(\mathbb{Q})$, $C\in M_{2,n-2}(\mathbb{Z})$, $H\in M_{2}(\mathbb{Q})$ and $D=\operatorname{diag}(d_1,...,d_{n-2})$
\end{lemma}
\begin{proof}
The proof goes along the same lines as before. Note that each vector $(a_{1,2+j},a_{2,2+j})^T$ for $j=1,2,...,n-2$ is a linear combination of $(a_{1,3},a_{2,3})^T$ and $(a_{1,4},a_{2,4})^T$. Therefore we have
\begin{center}
$
    A= \begin{pmatrix}
A_1&(\bm{\gamma}_1,\bm{\gamma}_2)C\\
C^T(\bm{\gamma}_1,\bm{\gamma}_2)^T&A_2
\end{pmatrix}
$
\end{center}
after rescaling we get $C\in M_{2,n-2}(\mathbb{Z})$ and $\operatorname{rank}((\bm{\gamma}_1,\bm{\gamma}_2))=2$. For the final part let's define for $3\leq j\leq n$ and $1\leq i\leq 2$, $\bm{\eta}_j=(a_{j,3},...,a_{j,j-1},a_{j,j+1},...,a_{j,n})^T$ and $\bm{\theta}_{i,j}=(a_{i,3},...,a_{i,j-1},a_{i,j+1},...,a_{i,n})^T$. Similarly as before we can say that for each $j$, $\bm{\eta}_j$ is a linear combination of $\bm{\theta}_{1,j},\bm{\theta}_{2,j}$. For $j=3$ it's because $\operatorname{rank}(B_2)=2$ for $j=4$ because $\operatorname{rank}(B_1)=2$ and for $j\geq 5$ because $\operatorname{rank}(B)=2$. Finally for $i=1,2$ we define $\bm{\theta}_i=(a_{i,3},...,a_{i,n})^T$ and for $3\leq j\leq n$ we observe that there is a unique $a_{j,j}'\in \mathbb{Q}$ such that  $\bm{\varphi}_j=(a_{j,3},...,a_{j,j-1},a_{j,j}',a_{j,j+1},...,a_{j,n})^T$ is a linear combination of $\bm{\theta}_1$ and $\bm{\theta}_2$. This means that $A$ has the form
\begin{center}
$
    A= \begin{pmatrix}
A_1&(\bm{\gamma}_1,\bm{\gamma}_2)C\\
C^T(\bm{\gamma}_1,\bm{\gamma}_2)^T&D+MC
\end{pmatrix}
$
\end{center}
and as $MC$ is symmetric it can be written in the form $C^THC$ for
$H\in M_2(\mathbb{Q})$.
\end{proof}
\begin{lemma}
Let $A$ be given by the form above. Let $C=(\bm{\xi}_1,\bm{\xi}_2,...,\bm{\xi}_{n-2})$. Then we can find pairwise distinct $b_i$ with $1\leq b_1,b_2,b_3,b_4,b_5\leq n-2$ such that $\operatorname{rank}((\bm{\xi}_{b_1},\bm{\xi}_{b_2}))=2$ and $d_{b_3}d_{b_4}d_{b_5}\neq 0$
\end{lemma}
\begin{proof}
We know that as $\operatorname{rank}(B)=2$ we have $\operatorname{rank}(\bm{\xi}_1,\bm{\xi}_2)=2$. So suppose the above statement doesn't happen. Namely $\forall k_1,k_2,k_3\geq 3$ we have $d_{k_1}d_{k_2}d_{k_3}=0$. This gives the upper bound $\operatorname{rank}(D)\leq 4$. But we know from looking at form of $A$ and the fact that $\operatorname{rank}(A)\geq 8$ we have $\operatorname{rank}(D)\geq 4$.\\
\indent So we only have case $\operatorname{rank}(D)=4$ left. We may suppose that $d_1d_2d_3d_4\neq 0$ and $d_j=0$ for $j\geq 4$. But then by looking at the rank of whole matrix we conclude that $\operatorname{rank}(\bm{\xi}_5,\bm{\xi}_6,...,\bm{\xi}_{n-2})=2$. That then obviously gives us a desired quintuple.
\end{proof}
\begin{proof}[ \,Proof of Lemma~{\upshape\ref{lem3}}]
Using Lemma ~{\upshape\ref{lem8}} we know the structure of our exponential sum.
\begin{equation}
       S(\alpha)=\sum_{\substack{ \textbf{x}\leq X \\ \textbf{x}\in \mathbb{N}^2}}
    \sum_{\substack{ \textbf{y}\leq X \\ \textbf{y}\in \mathbb{N}^{n-2}}}\La(\xv) \La(\yv)
     e(\alpha(\xt A_1\xv +2\xt( \bm{\gamma}_1,\bm{\gamma}_2)C\yv+\yv^TD\yv+\yv^TC^THC^T\yv))
\end{equation}
For each $\yv\in \mathbb{N}^{n-2}$ satisfying $\yv\leq X$  we let $\zv=C\yv$ and use orthogonality of exponential function to get
\begin{equation}
    \begin{aligned}
S(\alpha)={} &\int_{[0,1]^2}
    \sum_{\substack{\lvert\textbf{z}\rvert\ll X \\ \textbf{z}\in \mathbb{Z}^2}}
    \sum_{\substack{ \textbf{x}\leq X \\ \textbf{x}\in \mathbb{N}^2}}
    \sum_{\substack{ \textbf{y}\leq X \\ \textbf{y}\in \mathbb{N}^{n-2}}}\La(\xv)\La(\yv)
     e(\alpha(\xt A_1\xv +2\xt( \bm{\gamma}_1,\bm{\gamma}_2)\zv+\yv^TD\yv)\\
     & \times e(\al \zv^TH\zv+ (\yv^TC^T-\zv^T)\be)\, d\bm{\beta}.
    \end{aligned}
\end{equation}
We introduce new sums
\begin{equation}
    \mathcal{F}(\alpha,\bm{\beta})=  \sum_{\substack{\lvert\textbf{z}\rvert\ll X \\ \textbf{z}\in \mathbb{Z}^2}}  \sum_{\substack{ \textbf{x}\leq X \\ \textbf{x}\in \mathbb{N}^2}} \La(\xv)   e(\alpha(\xt A_1\xv+2\xt(\bm{\gamma}_1,\bm{\gamma}_2)\zv+\zv^TH\zv)-\zv^T\be)
\end{equation}
and for $i=1,...,n-2$
\begin{equation}
  f_i(\al,\be)=  \sum_{\substack{y_i\leq X \\ y_i\in \mathbb{N}}}\La(y_i)e(\al d_iy_i^2 +y_i\bm{\xi}_i^T\be)
\end{equation}
where $\bm{\xi}_i$ is the $i$-th column vector of $C$.
We again use CS inequality to get that:
\begin{equation}
    \begin{aligned}
     \int_{\mathfrak{m}} \abs{S(\alpha )}\,d\alpha\leq 
    &\Bigg(\sup_{\substack{\alpha\in \m \\ \bm{\beta}\in [0,1]^2}}\big\vert\prod_{k\neq b_1,b_2,b_3,b_4} f_k(\alpha,\bm{\beta})\big\vert\Bigg)\displaystyle\left(\int_{[0,1]^3}\abs{\mathcal{F}(\alpha,\be)}^2 \,d\al \,d\be \right)^{\tfrac{1}{2}}\\
    &\times\displaystyle\left(\int_{[0,1]^3}\abs{f_{b_1}(\alpha,\bm{\beta})f_{b_2}(\alpha,\bm{\beta})f_{b_3}(\alpha,\bm{\beta})f_{b_4}(\alpha,\bm{\beta})}^2\,d\al \,d\be \right)^{\tfrac{1}{2}}        
    \end{aligned}
\end{equation}
where $b_i$ are chosen according to the rules from the previous lemma. 
The supremum can be bounded using Lemma 3.1 (as $d_{b_5}\neq 0$) by 
\begin{equation}
    \sup_{\substack{\alpha\in \m \\ \bm{\beta}\in [0,1]^2}}\abs{\prod_{k\neq b_1,b_2,b_3,b_4} f_k(\alpha,\bm{\beta})}\ll X^{n-6}L^{-K}.
\end{equation}
First integral can be bounded using Lemma 3.2 and 3.3 by
\begin{equation}
    \ll L^{4} \Bigg\{\substack{2\xt A_1\xv'+ 2\xt (\bm{\gamma}_1,\bm{\gamma}_2)\zv'+2\zv^T (\bm{\gamma}_1,\bm{\gamma}_2)^T\xv'+2\zv^TH\zv'=0\\ 
    z_1=0 \\ z_2=0}\Bigg\}.
\end{equation}
After plugging in $z_1=z_2=0$ we are left with 
\begin{equation}
    \ll L^{4} \Bigg\{2\xt A_1\xv'+ 2\xt (\bm{\gamma}_1,\bm{\gamma}_2)\zv'=0\Bigg\}\ll L^5X^4
\end{equation}
where we again used Lemma 3.4 as we're working with 6 variables and $(\bm{\gamma}_1,\bm{\gamma}_2)$ is invertible. If we look at the second integral and use Lemma 3.2 and 3.3 we obtain the bound:
\begin{equation}
    \ll L^{8} \Bigg\{\substack{d_{b_1}y_{b_1}y_{b_1}'+d_{b_2}y_{b_2}y_{b_2}'+d_{b_3}y_{b_3}y_{b_3}'+d_{b_4}y_{b_4}y_{b_4}'=0\\ y_{b_1}\bm{\xi}_{b_1}+y_{b_2}\bm{\xi}_{b_2}+y_{b_3}\bm{\xi}_{b_3}+y_{b_4}\bm{\xi}_{b_4}=0}\Bigg\}.
\end{equation}
Because $\operatorname{rank}((\bm{\xi}_{b_1},\bm{\xi}_{b_2}))=2$ we can write $y_{b_1}$ and $y_{b_2}$ in terms of $y_{b_3}$ and $y_{b_4}$ and plug in into the first equation. We are left with bilinear equation in 6 variables and as $d_{b_3}d_{b_4}\neq 0$ it has rank at least 2. So after applying Lemma 3.4 we get
\begin{equation}
    \left(\int_{[0,1]^3}\abs{f_{b_1}(\alpha,\bm{\beta})f_{b_2}(\alpha,\bm{\beta})f_{b_3}(\alpha,\bm{\beta})f_{b_4}(\alpha,\bm{\beta})}^2\,d\al \,d\be \right)^{\tfrac{1}{2}} \ll (L^9 X^4)^{1/2}=L^{9/2}X^2
\end{equation}
Once we combine all three estimated we get the desired result and finishes the whole discussion on off diagonal rank 2 matrices.
\end{proof}
\section{Some Thanks}
I would like to thank my supervisor Aled Walker for an indescribable support and Department of Pure Mathematics and Mathematical Statistics for funding my work on this project.

\end{document}